\documentclass[11pt]{article}
\usepackage{enumerate}
\usepackage{amssymb,a4wide,latexsym,makeidx,epsfig,fleqn}
\usepackage{amsthm}
\usepackage{amsmath}
\usepackage{enumerate}
\usepackage{extarrows}
\usepackage{graphicx}
\usepackage{subfigure}
\usepackage{float}
\usepackage[justification=centering]{caption}
\newtheorem{theorem}{Theorem}[section]
\newtheorem{remark}[theorem]{Remark}

\newtheorem{definition}[theorem]{Definition}
\newtheorem{lemma}[theorem]{Lemma}

\newtheorem{corollary}[theorem]{Corollary}

\newtheorem{conjecture}[theorem]{Conjecture}

\begin{document}
\textwidth 150mm \textheight 225mm
\title{ The generalized distance matrix of digraphs
\thanks{Supported by the National Natural Science Foundation of China (No. 11871398), the Natural Science Basic Research Plan in Shaanxi Province of China (Program No. 2018JM1032) and China Scholarship Council (No. 201706290182).
 %\vskip 0.05in The author is currently a visiting Ph.D. student at the Department of Mathematics
 %and Statistics in San Jose State University from September 2017 to September 2018.
 }}
\author{{Weige Xi$^{1,2}$, Wasin So$^3$, Ligong Wang$^{1,2,}\footnote{Corresponding author.}$ }\\
{\small $^1$ Department of Applied Mathematics, School of Science,}\\ {\small  Northwestern Polytechnical University, Xi¡¯an, Shaanxi 710072, China.}\\
{\small $^2$ Xi'an-Budapest Joint Research Center for Combinatorics,}\\ {\small Northwestern Polytechnical University, Xi'an, Shaanxi 710129, China.}\\
{\small $^3$ Department of Mathematics and Statistics, San Jose State University, }\\
{\small San Jose, CA 95192-0103, USA.}\\
{\small E-mail:  xiyanxwg@163.com, wasin.so@sjsu.edu, lgwangmath@163.com }\\}

\date{}
\maketitle
\begin{center}
\begin{minipage}{120mm}
\vskip 0.3cm
\begin{center}
{\small {\bf Abstract}}
\end{center}
{\small  \ Let $D(G)$ and $D^Q(G)= Diag(Tr) + D(G)$ be the distance matrix
and distance signless Laplacian matrix of a simple strongly connected digraph $G$, respectively,
where $Diag(Tr)=\textrm{diag}(D_1,D_2,$ $\ldots,D_n)$
be the diagonal matrix with vertex transmissions of the digraph $G$. To track the gradual change
of $D(G)$ into $D^Q(G)$, in this paper, we propose to study the convex combinations of $D(G)$ and $Diag(Tr)$ defined by
$$D_\alpha(G)=\alpha Diag(Tr)+(1-\alpha)D(G), \ \ 0\leq \alpha\leq1.$$
This study reduces to merging the distance spectral and distance signless
Laplacian spectral theories. The eigenvalue with the largest modulus
of $D_\alpha(G)$ is called the $D_\alpha$ spectral radius of $G$, denoted by
$\mu_\alpha(G)$. We determine the digraph which attains the maximum (or minimum)
$D_\alpha$ spectral radius among all strongly connected digraphs. Moreover, we also determine the digraphs which attain the minimum
$D_\alpha$ spectral radius among all strongly connected digraphs with given parameters such as dichromatic number, vertex connectivity or arc connectivity.

 }
\end{minipage}
\end{center}

\section{Introduction }

A digraph is simple if it has no loops and multiple arcs. A digraph is strongly
connected if for every pair of vertices $v_i, v_j\in V(G)$, there
exists a directed path from $v_i$ to $v_j$. Throughout this paper, we only
consider simple strongly connected digraphs. We
use standard terminology and notation, and refer the reader to \cite{BG} for an extensive treatment of digraphs.

Let $G=(V(G), E(G))$ be a simple strongly connected digraph, if there is an arc from $v_i$ to $v_j$, we indicate this by writing
$(v_i,v_j)$, call $v_j$ (resp. $v_i$) the head (resp. the tail)
of $(v_i,v_j)$, the vertex $v_i$ is a tail of $v_j$, and $(v_i,v_j)$ is said to be out-incident to $v_i$ and in-incident to $v_j$.
For any vertex $v_i$,
let $N_i^{+}=N^+_{v_i}(G)=\{v_j\in V(G) \mid (v_i,v_j)\in
E(G)\}$ and $N_i^{-}=N^-_{v_i}(G)=\{v_j\in V(G)
\mid (v_j,v_i)\in E(G)\}$ denote the out-neighbors
and in-neighbors of $v_i$, respectively. Let $d_i^+=d_{v_i}^+(G)=|N_i^{+}|$
denote the outdegree of the vertex $v_i$, and $d_i^-=d_{v_i}^-(G)=|N_i^{-}|$
denote the indegree of the vertex $v_i$ in $G$. Let $C_{n}$ denote the directed cycle on $n$ vertices.
Let $\overset{\longleftrightarrow}{K_{n}}$
denote the complete digraph on $n$ vertices in which two arbitrary distinct vertices
$v_i,v_j\in V(\overset{\longleftrightarrow}{K_{n}})$, there are arcs
$(v_i,v_j)\in E(\overset{\longleftrightarrow}{K_{n}})$ and $(v_j,v_i)\in E(\overset{\longleftrightarrow}{K_{n}})$.
A tournament is a directed graph obtained by assigning a direction for each edge
in an undirected complete graph. A digraph is a transitive tournament if it is
tournament and the set of its outdegrees is $\{0,1,2,\ldots,n-1\}$.

Let $G$ be a digraph. If $S\subset V(G)$, then we use $G[S]$ to denote the subdigraph of $G$ induced by $S$. We use $G+e$ to denote
the digraph obtained from $G$ by adding the arc $e\notin E(G)$. Let $G_1$ and $G_2$ be two disjoint digraphs, $G_1\cup G_2$ is a digraph with vertex set $V(G_1)\cup V(G_2)$ and arc set $E(G_1)\cup E(G_2)$.
We denote by $G_1\vee G_2$ the join of $G_1$ and $G_2$, which is the digraph such that $V(G_1\vee G_2) =V(G_1)\cup V(G_2)$
and $E(G_1\vee G_2) =E(G_1)\cup E(G_2)\cup\{(u,v),(v,u): u\in V(G_1) \ \textrm{and} \ v\in V(G_2)\}$.

A digraph is acyclic if it has no directed cycle.
For a digraph $G$, a vertex set $F\subset V(G)$ is acyclic if its induced subdigraph $G[F]$
is acyclic. A partition of $V(G)$ into $k$ acyclic sets is called a
$k$-coloring of $G$. The minimum integer $k$ for which there exists
a $k$-coloring of $G$ is the dichromatic number $\chi(G)$ of the digraph $G$. For a strongly connected digraph $G=(V(G),E(G))$,
the vertex connectivity of $G$, denoted by $\kappa(G)$, is the minimum
number of vertices whose removal destroys the strongly connectivity of $G$. A
set of arcs $S\subset E(G)$ is an arc cut set if $G-S$ is not strongly connected.
The arc connectivity of $G$, denoted by $\kappa'(G)$, is the minimum
number of arcs whose deletion yields the resulting digraph non-strongly connected.

For a simple strongly connected digraph $G$ of order $n$, let $D(G)=(d_{ij})$ be the distance matrix of $G$, where $d_{ij}=d_G(v_i, v_j )$
be the length of shortest directed path from
$v_i$ to $v_j$ in $G$. We call $D_i=D_{v_i}(G) =\sum\limits^n_{j=1}d_{ij}$
the transmission of the vertex $v_i$ of $G$. We call a digraph $G$ $r$-distance regular if
$D_1 = D_2 =\ldots= D_n=r$. Let $Diag(Tr)=\textrm{diag}(D_1, D_2,\ldots, D_n)$ be
the diagonal matrix with vertex transmissions of $G$. Then $D^Q(G)= Diag(Tr) + D(G)$ is called the distance signless Laplacian matrix of $G$.
The spectral radius of $D(G)$ is called the distance spectral radius of $G$, and the spectral radius of $D^Q(G)$ is called
the distance signless Laplacian spectral radius of $G$.

So far, the distance spectrum and distance signless Laplacian spectrum of connected undirected graphs have been investigated extensively,
see \cite{AP,AP1,LZ,XZ,XZL} as well as the references therein. Recently, the distance spectral radius and distance
signless Laplacian spectral radius of digraphs have also been studied in some papers. For example, Lin et al. \cite{LS} characterized the extremal digraphs with minimum distance spectral radius among all digraphs with given vertex connectivity.
Lin and Shu \cite{LiS} characterized the digraphs having the maximal and minimal distance spectral radii among all strongly connected digraphs,
and they also determined the extremal digraphs which have the minimal distance spectral radius among all strongly connected digraphs
with given arc connectivity and dichromatic number, respectively. Li et al. \cite{LWM}
determined the extremal digraph which has the minimum distance
signless Laplacian spectral radius among all strongly connected digraphs
with given dichromatic number. Li et al. \cite{LWM1}, Xi and Wang \cite{XiWa2} independently
determined the extremal digraph with the minimum distance signless Laplacian spectral radius among
all strongly connected digraphs with given vertex connectivity.

Nikiforov \cite{Niki} proposed to study the convex linear combinations of the adjacency matrix and diagonal degree matrix of undirected graphs,
which can underpin a unified theory of adjacency spectral and signless Laplacian spectral theories. Similarly, Cui et al. \cite{CHT}
proposed to study the convex combinations of the distance matrix and
the diagonal matrix with vertex transmissions of undirected graphs, which reduces to merging the distance spectral and distance signless
Laplacian spectral theories. Here, we extend the definition to digraphs. We consider to study the convex combinations $D_\alpha(G)$ of $Diag(Tr)$ and $D(G)$
defined by
$$D_\alpha(G)=\alpha Diag(Tr)+(1-\alpha)D(G), \ \ 0\leq \alpha\leq1.$$
Many facts suggest that the study of the family $D_\alpha(G)$ is long due. To begin with,
obviously,
$$D(G)=D_0(G), \ \ \  Diag(Tr)(G)=D_1(G), \ \ \ \textrm{and} \ \ \ D^Q(G)=2D_{\frac{1}{2}}(G).$$
Since $D_{\frac{1}{2}}(G)$ is essentially equivalent to $D^Q(G)$, in this paper
we take $D_{\frac{1}{2}}(G)$ as
an exact substitute for $D^Q(G)$. With this caveat, one can see that $D_\alpha(G)$ seamlessly
joins $D(G)$ to $Diag(Tr)$ with $D^Q(G)$ being right in the middle of the range. The spectral radius of
$D_\alpha(G)$, i.e., the eigenvalues with largest modulus of $D_\alpha(G)$, is called the $D_\alpha$ spectral radius of $G$, denoted by
$\mu_\alpha(G)$. The main goal of this paper is to extend some results
on maximum or minimum $D_0$ spectral radius and $D_{\frac{1}{2}}$ spectral radius for all $\alpha\in[0,1)$.

If $\alpha=1$, $D_1(G)=Diag(Tr)$ the diagonal matrix with vertex transmissions of $G$ which is not interesting.
So we only consider the cases $0\leq \alpha<1$ for the rest of this paper. If $G$ is a
strongly connected digraph, then it follows from the Perron Frobenius Theorem \cite{HJ} that $\mu_\alpha(G)$ is an eigenvalue of $D_\alpha(G)$,
and there is a unique positive unit eigenvector corresponding to $\mu_\alpha(G)$. The positive unit eigenvector
corresponding to $\mu_\alpha(G)$ is called the Perron vector of $D_\alpha(G)$.

The rest of the paper is structured as follows. In the next section we introduce
some lemmas and give basic facts about the $D_\alpha$ spectral radius of a digraph $G$. In Section 3,
we characterize the unique digraph which has the minimum $D_\alpha$ spectral radius among all strongly connected
digraphs with given dichromatic number. In Section 4, we characterize the extremal digraphs
which attain the minimum $D_\alpha$ spectral radius among all strongly connected digraphs with given
vertex connectivity. In Section 5, we characterize the extremal digraphs with the minimum $D_\alpha$ spectral radius among all strongly connected digraphs with given
arc connectivity.

\section{Preliminaries and basic properties of $\mu_\alpha(G)$}

In the rest of this section, let $\rho(\cdot)$ denote the spectral radius of a square matrix.

\noindent\begin{lemma}\label{le:1} (\cite{HJ}) \ Let $M=(m_{ij})$
be an $n \times n$ nonnegative irreducible matrix, and $R_i(M)$ be the $i$-th row sum of $M$. Then
$$\min\{R_{i}(M):1 \leq i\leq n\}\leq \rho(M)\leq \max\{R_{i}(M):1
\leq i\leq n\}.$$
 Moreover, either one equality holds  if
and only if $R_1(M)=R_2(M)=\ldots=R_n(M)$.
\end{lemma}

Applying Lemma \ref{le:1} to digraphs, we have the following result.

\noindent\begin{theorem}\label{th:1} \ Let $G$ be a strongly connected digraph with $V(G)=\{v_1, v_2, \ldots, v_n\}$. Then
$$\min\{\alpha D_i+\frac{1-\alpha}{D_i}T_i : 1
\leq i\leq n\}\leq \mu_\alpha(G)\leq \max\{\alpha D_i+\frac{1-\alpha}{D_i}T_i : 1
\leq i\leq n\},$$
where $T_i=\sum\limits_{t=1}^n d_{it}D_t$. Moreover, if $\frac{1}{2}<\alpha<1$, then either one equality holds if
and only if $G$ is distance regular.
\end{theorem}

\begin{proof} Let $Diag(Tr)=\textrm{diag}(D_1, D_2,\ldots, D_n)$ be
the diagonal matrix with vertex transmissions of $G$. By a simple calculation, we get the $i$-th row sum of $Diag(Tr)^{-1}D_\alpha(G)Diag(Tr)$ is
$$R_i(Diag(Tr)^{-1}D_\alpha(G)Diag(Tr))=\alpha D_i+\frac{1-\alpha}{D_i}\sum\limits_{t=1}^n d_{it}D_t=\alpha D_i+\frac{1-\alpha}{D_i}T_i.$$
Take $M=Diag(Tr)^{-1}D_\alpha(G)Diag(Tr)$, by Lemma \ref{le:1}, the required result follows.

For $\frac{1}{2}<\alpha<1$, suppose that either of the equalities holds. Then Lemma \ref{le:1} implies that the row sums of $Diag(Tr)^{-1}D_\alpha(G)Diag(Tr)$ are all equal.
That is, for any vertices $v_i,v_j\in V(G)$,
$$\alpha D_i+\frac{1-\alpha}{D_i}T_i=\alpha D_j+\frac{1-\alpha}{D_j}T_j.$$

We use $Tr_{\max}$ and $Tr_{\min}$ denote the maximum and minimum vertex transmission of $G$, respectively.
Without loss of generality, assume that $D_1=Tr_{\max}$ and $D_n=Tr_{\min}$. It is easy to know that
$T_1=\sum\limits_{t=1}^n d_{1t}D_t\geq D_1D_n=Tr_{\max} Tr_{\min}$ and $T_n=\sum\limits_{t=1}^n d_{nt}D_t\leq D_1D_n=Tr_{\max} Tr_{\min}$. Thus, we obtain
$$\alpha Tr_{\max}+(1-\alpha)Tr_{\min}\leq\alpha Tr_{\max}+\frac{1-\alpha}{Tr_{\max}}T_1=\alpha Tr_{\min}+\frac{1-\alpha}{Tr_{\min}}T_n\leq\alpha Tr_{\min}+(1-\alpha)Tr_{\max}.$$
which implies that $Tr_{\max}=Tr_{\min}$ for $\frac{1}{2}<\alpha<1$. Therefore, $G$ is distance regular.

Conversely, if $G$ is a $r$-distance regular digraph, then $\mu_\alpha(G) =r$. On the other hand, by a simple calculation, we get
$$\alpha D_i+\frac{1-\alpha}{D_i}T_i=r$$
for any $v_i\in V(G)$. Therefore, both of the equalities hold.
\end{proof}

\begin{definition}\label{de:1} (\cite{BP,HJ}) Let $A = (a_{ij})$ and $B = (b_{ij})$ be $n\times m$ matrices.
If $a_{ij} \leq b_{ij}$ for all $i$ and $j$, then $A \leq B$. If $ A \leq B$ and $A\neq B$, then $A < B$.
If $a_{ij}< b_{ij}$ for all $i$ and $j$, then $A\ll B$.
\end{definition}

\begin{lemma}\label{le:2} (\cite{BP,HJ}) Let $A$ and $B$ be nonnegative matrices. If $0 \leq A \leq B$,
then $\rho(A) \leq \rho(B)$. Furthermore, if $B$ is irreducible and $0\leq A< B$, then $\rho(A)<\rho(B)$.
\end{lemma}

By Lemma \ref{le:2}, we have the following result in terms of $D_\alpha$ spectral radius of digraphs.

\begin{corollary}\label{co:1} Let $G$ be a strongly connected digraph with $u, v\in V(G)$ and $(u,v)\notin E(G)$.
Then $\mu_\alpha(G)>\mu_\alpha(G+(u,v))$.
\end{corollary}

By Lemma \ref{le:1} and Corollary \ref{co:1}, we have the following theorem.

\begin{theorem}\label{th:2} Let $G$ be a strongly connected digraph. Then
$n-1\leq\mu_\alpha(G)\leq \frac{n(n-1)}{2}$, $\mu_\alpha(G)=n-1$ if and only if $G\cong\overset{\longleftrightarrow}{K_{n}}$, and
$\mu_\alpha(G)=\frac{n(n-1)}{2}$ if and only if $G\cong C_{n}$.
\end{theorem}

\noindent\begin{lemma}\label{le:3} (\cite{BP}) Let $B$ be a nonnegative irreducible matrix and
$X=(x_1,x_2,$ $\ldots,x_n)^T$ be any nonzero nonnegative vector. If $\beta,\gamma\geq0$ such that $\beta X< BX<\gamma X$, then $\beta<\rho(B)<\gamma$.
\end{lemma}

By Lemma \ref{le:3}, we have the following result in terms of $D_\alpha$ spectral radius of digraphs.

\begin{corollary}\label{co:c3} Let $G$ be a strongly connected digraph with $V(G)=\{v_1, v_2, \ldots, v_n\}$. Then
$\lambda_\alpha(G)> \alpha Tr_{\max}$, where $Tr_{\max}$ denote the maximum vertex transmission of $G$
\end{corollary}

\begin{proof} With loss of generality, let $D_1=Tr_{\max}$ is the maximum vertex transmission.
Taking $X=(1,0,\ldots,0,0,0\ldots,0)^T$,
that is, all the entries of $X$ are 0 except $x_1=1$, where $x_1$ corresponding to the vertex $v_1$. Since
$G$ is strongly connected, $d_{j1}\geq1$ for all $v_j\neq v_1$ and $v_j\in V(G)$.
Thus $D_\alpha(G)X>\alpha D_1X=\alpha Tr_{\max} X$.
Therefore, by Lemma \ref{le:3}, we have $\mu_\alpha(G)>\alpha Tr_{\max}$.
\end{proof}

\section{The minimum $D_\alpha$ spectral radius of strongly connected digraphs with given dichromatic number}

Let $\mathcal{C}_n^k$ denote the set of strongly connected digraphs
with order $n$ and dichromatic number $\chi(G)=k\geq2$. Let $[V_1,V_2]$ denote the arcs between $V_1$ and $V_2$.
Let $\mathcal{T}_n^k$ denote the set of digraphs of order $n$ with $V(\mathcal{T}_n^k)=V^1\cup V^2\ldots\cup V^k$,
where $V^i$ $(i=1,2,\ldots,k)$ is a transitive tournament
and $[V^i,V^j]=\{(v^i_s,v^j_t), (v^j_t,v^i_s): v^i_s\in V^i, v^j_t\in V^j\}$ for all $i\neq j$ and $i,j\in\{1,2,\ldots,k\}$.
Let $\mathcal{T}_n^{k*}$ denote the digraph in $\mathcal{T}_n^k$
with $||V^i|-|V^j||\leq1$ for all $i,j\in\{1,2,\ldots,k\}$.

In \cite{LiS}, Lin and Shu proved that $\mathcal{T}_n^{k*}$ attains the minimal $D_0$ spectral radius
among all strongly connected digraphs with given dichromatic number.
In \cite{LWM}, Li et al. determined that $\mathcal{T}_n^{k*}$ also attains the minimum $D_{\frac{1}{2}}$
spectral radius among all strongly connected digraphs with given dichromatic number.
We generalize their results to $0\leq\alpha<1$. In the rest of this section,
 we will show that $\mathcal{T}_n^{k*}$ achieves the
minimum $D_\alpha$ spectral radius among all digraphs in $\mathcal{C}_n^k$.

\noindent\begin{lemma}\label{le:10} (\cite{BoMu}) Let $G$ be a digraph with
no directed cycle. Then $\delta^-=0$, where $\delta^-$ denotes the minimum
indegree of $G$, and there is an ordering
$v_1, v_2, \ldots, v_n$ of $V(G)$ such that, for $1 \leq i \leq n$,
every arc of $G$ with head $v_i$ has its tail in $\{v_1, v_2, \ldots, v_{i-1}\}$.
\end{lemma}

Let $G$ be a strongly connected digraph with order $n$ and dichromatic
number $\chi(G)=k\geq2$. By the definition, $G$ has $k$-color classes
and each is an acyclic set. Suppose the $k$-color classes are
$V^1, V^2,\ldots, V^k$ with orders $n_1, n_2,\ldots, n_k$, respectively. Without
loss of generality, we suppose that $n_1\geq n_2\geq\ldots\geq n_k$.
By Corollary \ref{co:1}, we know that the addition of arcs will decrease the
$D_\alpha$ spectral radius. Then, by Lemma \ref{le:10}, we know that the digraphs which achieve the minimum $D_\alpha$
spectral radius must be the digraphs in $\mathcal{T}_n^k$. Next we will prove that
the digraph $\mathcal{T}_n^{k*}$ has the minimum $D_\alpha$ spectral radius in $\mathcal{T}_n^k$.

\noindent\begin{theorem}\label{th:c-3} Let $G\in\mathcal{T}_n^k$, with
the $k$-color classes $V^1, V^2,\ldots, V^k$ satisfying $n_1\geq n_2\geq\ldots\geq n_k$,
where $n_i=|V^i|$ and $V^i$ is a transitive tournament for each $i\in\{1,2,\ldots,k\}$.
Then $\mu_\alpha(G)\geq \mu_\alpha(\mathcal{T}_n^{k*})$ with equality holds if and only if
$G\cong\mathcal{T}_n^{k*}$.
\end{theorem}

\begin{proof} Let $G$ be an arbitrary digraph in
$\mathcal{T}_n^k$ and $\mu_\alpha(G)=\mu$ be the $D_\alpha$ spectral radius of $G$.
Since each $V^i$ is a transitive tournament, we can give a vertex ordering
$\{v^i_1, v^i_2,\ldots, v^i_{n_i}\}$ such that $(v^i_s,v^i_t)\in E(G)$, for all $s<t$.
Thus we suppose that $X= (x^1_1, x^1_2,\ldots, x^1_{n_1},
x^2_1,x^2_2,$ $\ldots, x^2_{n_2},\ldots,x^k_1, x^k_2,\ldots, x^k_{n_k})^T$
is the Perron vector of $D_\alpha(G)$ corresponding to $\mu=\mu_\alpha(G)$, where $x^i_j$
 corresponds to $v^i_j$ for each $1 \leq i \leq k$ and $1 \leq j \leq n_i$. Then
we have the following two claims.
\\{\bf Claim 1.} $x^i_1< x^i_2<\ldots<x^i_{n_i}$ ($1\leq i\leq k$).
\\From $D_\alpha(G)X=\mu X$, we have
\begin{align}
\mu x^i_j&=\alpha D_{v^i_j}x^i_j+(1-\alpha)\sum\limits^{n_i}_{t=j+1}x^i_t+2(1-\alpha)\sum\limits^{j-1}_{t=1}x^i_t+(1-\alpha)L\notag\\
&=\alpha D_{v^i_j}x^i_j+(1-\alpha)x^i_{j+1}+(1-\alpha)\sum\limits^{n_i}_{t=j+2}x^i_t+2(1-\alpha)\sum\limits^{j-1}_{t=1}x^i_t+(1-\alpha)L.
\end{align}

\begin{align}
\mu x^i_{j+1}&=\alpha D_{v^i_{j+1}}x^i_{j+1}+(1-\alpha)\sum\limits^{n_i}_{t=j+2}x^i_t+2(1-\alpha)x^i_j+2(1-\alpha)\sum\limits^{j-1}_{t=1}x^i_t+
(1-\alpha)L,
\end{align}
where $L=\sum\limits^k_{\scriptstyle m=1\atop\scriptstyle m\neq i}\sum\limits^{n_m}_{t=1}x^m_t$,
$D_{v^i_j}=\sum\limits^k_{\scriptstyle m=1 \atop \scriptstyle m\neq i} n_m + (n_i-j)+2(j-1)=n+j-2$,
$D_{v^i_{j+1}}=\sum\limits^k_{\scriptstyle m=1 \atop \scriptstyle m\neq i} n_m + n_i-(j+1)+2j=n+j-1$.

Subtracting (1) from (2), we get
$\mu (x^i_{j+1}-x^i_j)=(\alpha(n+j)-1)(x^i_{j+1}-x^i_j)+x^i_{j}$. By Corollary \ref{co:c3}, we have
$\mu>\alpha D_{v^i_{j+1}}=\alpha(n+j-1)$. So we further have
$(\mu-\alpha(n+j)+1)(x^i_{j+1}-x^i_j)=x^i_{j}>0$, which implies
$x^i_{j+1}>x^i_j$. Therefore, Claim 1 holds.
\\{\bf Claim 2.} $x^1_i= x^2_i=\ldots=x^k_{i}, 1\leq i\leq n_k$.
\\We verify the claim by induction on $i$. If $i=1$, then we have
\begin{equation}\label{eq:3}
\mu x^p_1=\alpha D_{v^p_1}x^p_1+(1-\alpha)\sum\limits^{n_p}_{t=2}x^p_t+(1-\alpha)L_1
+(1-\alpha)\sum\limits^{n_q}_{t=1}x^q_t
\end{equation}
\begin{equation}\label{eq:4}
\mu x^q_1=\alpha D_{v^q_1}x^q_1+(1-\alpha)\sum\limits^{n_q}_{t=2}x^q_t+(1-\alpha)L_1
+(1-\alpha)\sum\limits^{n_p}_{t=1}x^p_t
\end{equation}
where $L_1=\sum\limits^k_{\scriptstyle m=1, \scriptstyle m\neq p \atop\scriptstyle m\neq q}\sum\limits^{n_m}_{t=1}x^m_t$,
$D_{v^p_1}=D_{v^q_1}=n-1$.

Subtracting \eqref{eq:4} from \eqref{eq:3}, we get
$\mu (x^p_1-x^q_1)=\alpha(n-1)(x^p_1-x^q_1)-(1-\alpha)(x^p_1-x^q_1)=(\alpha(n-1)-(1-\alpha))(x^p_1-x^q_1)$. By Corollary \ref{co:c3}, we have
$\mu>\alpha(n-1)>\alpha(n-1)-(1-\alpha)$. This shows that
$x^p_1=x^q_{1}$ for all $1\leq p\neq q\leq k$.

Now we suppose that it holds for all $i< N\leq n_k$, that is,
$x^1_t= x^2_t=\ldots=x^k_{t}$, for each $t\in\{1,2,\ldots,N-1\}$. Next we
will consider the case when $i=N$, we have
\begin{equation}\label{eq:5}
\mu x^p_N=\alpha D_{v^p_N}x^p_N+(1-\alpha)\sum\limits^{n_p}_{t=N+1}x^p_t+2(1-\alpha)\sum\limits^{N-1}_{t=1}x^p_t+(1-\alpha)L_1
+(1-\alpha)\sum\limits^{n_q}_{t=1}x^q_t
\end{equation}
\begin{equation}\label{eq:6}
\mu x^q_N=\alpha D_{v^q_N}x^q_N+(1-\alpha)\sum\limits^{n_q}_{t=N+1}x^q_t+2(1-\alpha)\sum\limits^{N-1}_{t=1}x^q_t+(1-\alpha)L_1
+(1-\alpha)\sum\limits^{n_p}_{t=1}x^p_t
\end{equation}
where $L_1$ defined as the above, $D_{v^p_N}=D_{v^q_{N}}=n+N-2$.

Subtracting \eqref{eq:6} from \eqref{eq:5}, we get
$$\mu(x^p_N-x^q_N)=\alpha(n+N-2)(x^p_N-x^q_N)-(1-\alpha)\sum\limits^{N-1}_{t=1}x^q_t+(1-\alpha)x^q_N
+(1-\alpha)\sum\limits^{N-1}_{t=1}x^p_t-(1-\alpha)x^p_N.$$
By the inductive hypothesis, we further obtain that
$\mu(x^p_N-x^q_N)=\alpha(n+N-2)(x^p_N-x^q_N)-(1-\alpha)(x^p_N-x^q_N)$. By Corollary \ref{co:c3}, we have
$\mu>\alpha D_{v^p_N}=\alpha(n+N-2)$. This shows that
$x^p_N=x^q_{N}$ for all $1\leq p\neq q\leq k$. Therefore, Claim 2 holds.

If $G\ncong\mathcal{T}_n^{k*}$, then there exist $n_i$, $n_j$ such that
$|n_i-n_j|>1$. Without loss of generality, we suppose that $n_i\geq n_j+2$.
By Claim 1, we know that $x^i_1< x^i_2<\ldots<x^i_{n_i}$, so let
$G'=G-\{(v^i_{n_i},v^j_t): t=1,2,\ldots,n_j\}+\{(v^i_{n_i},v^i_l): l=1,2,\ldots,n_i-1\}$.
Then $G'\in\mathcal{T}_n^k$. For the Perron vector $X$ of $D_\alpha(G)$ corresponding to $\mu_\alpha(G)$ and for any
$t=1,2,\ldots,n$, we have
$$(D_\alpha(G) X)_t=\alpha D_t(G)x_t+(1-\alpha)\sum\limits_{s=1}^nd_{G}(v_t,v_s)x_s,$$
$$(D_\alpha(G') X)_t=\alpha D_t(G')x_t+(1-\alpha)\sum\limits_{s=1}^nd_{G'}(v_t,v_s)x_s.$$
From the above two equations we can observe that if $v_t\neq v^i_{n_i}$
then $(D_\alpha(G){\bf X})_t=(D_\alpha(G'){\bf X})_t$, otherwise
\begin{equation}\label{eq:7}
(D_\alpha(G) X)_t=\alpha D_{v^i_{n_i}}(G)x^i_{n_i}+(1-\alpha)L_3+(1-\alpha)\sum\limits^{n_j}_{t=1}x^j_t+2(1-\alpha)\sum\limits^{n_i-1}_{t=1}x^i_t,
\end{equation}
\begin{equation}\label{eq:8}
(D_\alpha(G')X)_t=\alpha D_{v^i_{n_i}}(G)x^i_{n_i}+(1-\alpha)L_3+(1-\alpha)\sum\limits^{n_i-1}_{t=1}x^i_t+2(1-\alpha)\sum\limits^{n_j}_{t=1}x^j_t
\end{equation}
where $L_3=\sum\limits^k_{\scriptstyle m=1, \scriptstyle m\neq i \atop\scriptstyle m\neq j}\sum\limits^{n_m}_{t=1}x^m_t$, $D_{v^i_{n_i}}(G)=n+n_i-2$,
$D_{v^i_{n_i}}(G')=n+n_j-1$.

By Claim 2, $x^1_i= x^2_i=\ldots=x^k_{i}, 1\leq i\leq n_k$.
Subtracting \eqref{eq:8} from \eqref{eq:7}, we get

\begin{align*}
(D_\alpha(G)X)_t-(D_\alpha(G')X)_t&=\alpha(n+n_i-2)x^i_{n_i}-\alpha(n+n_j-1)x^i_{n_i}+2(1-\alpha)\sum\limits^{n_i-1}_{t=1}x^i_t\\
&\ \ \ \ +(1-\alpha)\sum\limits^{n_j}_{t=1}x^j_t-2(1-\alpha)\sum\limits^{n_j}_{t=1}x^j_t
-(1-\alpha)\sum\limits^{n_i-1}_{t=1}x^i_t\\
&=\alpha(n_i-1-n_j)x^i_{n_i}+(1-\alpha)(\sum\limits^{n_i-1}_{t=1}x^i_t-\sum\limits^{n_j}_{t=1}x^j_t)\\
&=\alpha(n_i-1-n_j)x^i_{n_i}+(1-\alpha)\sum\limits^{n_i-1}_{t=n_j+1}x^i_t>0
\end{align*}
That is $\mu_\alpha(G)X=D_\alpha(G) X> D_\alpha(G')X$. Thus by Lemma \ref{le:3},
we have $\mu_\alpha(G)>\mu_\alpha(G')$.

We perform the above operation as many times as possible until there is no
$V^i, V^j$ such that $|n_i-n_j|\geq2$, which means the minimum $D_\alpha$
spectral radius of $\mathcal{T}_n^k$ is achieved only at the digraph $\mathcal{T}_n^{k*}$.
\end{proof}

Combining Lemma \ref{le:10} and Theorem \ref{th:c-3}, we have the following theorem.

\noindent\begin{theorem}\label{th:c-4} The digraph $\mathcal{T}_n^{k*}$
is the unique digraph which has the minimum $D_\alpha$ spectral radius among all
digraphs in $\mathcal{C}_n^k$.
\end{theorem}

\section{The minimum $D_\alpha$ spectral radius of strongly connected digraphs with given vertex connectivity}

Let $\mathcal{G}_{n,k}$ denote the set of strongly connected digraphs
with order $n$ and vertex connectivity $\kappa(G)=k\geq1$. If $k=n-1$,
then $\mathcal{G}_{n,k}=\{\overset{\longleftrightarrow}{K_{n}}\}$.
So we only consider the cases $1\leq k \leq n-2$.

For $1\leq m \leq n-k-1$, $K(n,k,m)$ denote the digraph $\overset{\longleftrightarrow}{K_k}\vee(\overset{\longleftrightarrow}{K_n}_{-k-m}\cup
\overset{\longleftrightarrow}{K_m})+ E$, where $E=\{(u,v): u\in V(\overset{\longleftrightarrow}{K_{m}}),
v\in V(\overset{\longleftrightarrow}{K_n}_{-k-m})\}$.
Let $\mathcal{K}(n,k)=\{K(n,k,m): \ 1\leq m \leq n-k-1\}$. Clearly $\mathcal{K}(n,k)\subset\mathcal{G}_{n,k}$.

In \cite{LS}, Lin et al. proved that $K(n,k,n-k-1)$ or $K(n,k,1)$ attains the minimum $D_0$ spectral radius among all digraphs with given vertex connectivity $k$.
The authors of \cite{LWM1} and \cite{XiWa2} independently
determined that $K(n,k,1)$ also attains the minimum $D_{\frac{1}{2}}$ spectral radius among
all strongly connected digraphs with given vertex connectivity $k$. We generalize their results to $0\leq\alpha<1$.

%\noindent\begin{lemma}\label{le:3} (\cite{HoYo})
%The digraph which achieved the maximum signless Laplacian spectral radius
% all strongly connected digraphs with vertex connectivity $1\leq\eta(G)=k\leq n-2$
% must be some digraph in $K(n,k)$.
%\end{lemma}

\noindent\begin{lemma}\label{le:11} (\cite{BoMu}) Let $G$ be a strongly
connected digraph with $\kappa(G)=k$. Suppose that $S$ is a $k$-vertex cut
of $G$ and $G_1,G_2,\ldots,G_t$ are the strongly connected components of $G-S$.
Then there exists an ordering of $G_1,G_2,\ldots,G_t$ such that for $1\leq i\leq t$
and any $v\in V(G_i)$, every tail of $v$ is in $\bigcup^i_{j=1}G_{j}$.
\end{lemma}

\noindent\begin{remark}\label{re:1}
By Lemma \ref{le:11}, we know that there exists a strongly
connected component of $G-S$, say $G_1$ with $|V(G_1)|=m$ such that
for any $v_i\in V(G_1)$, $|W^-_i|=0$, where $W_i^-=\{v_j\in V(G-S-G_1): (v_j,v_i)\in E(G)\}$.
Let $G_2=G-S-G_1$. We add arcs to $G$ until both induced subdigraph
of $V(G_1)\cup S$ and induced subdigraph of $V(G_2)\cup S$ attain to
complete digraphs, add arc $(u,v)$ for any $u\in V(G_1)$ and any $v\in V(G_2)$, the new resulting digraph denoted
by $H$. Since $G$ is $k$-strongly connected, then
$H=K(n,k,m)\in\mathcal{K}(n,k)\subset\mathcal{G}_{n,k}$. By Corollary \ref{co:1}, we have
$\mu_\alpha(G)\geq\mu_\alpha(H)$, with equality if and only if $G\cong H$. Therefore,
the digraphs which achieve the minimum $D_\alpha$ spectral radius among all digraphs
in $\mathcal{G}_{n,k}$ must be some digraphs in $\mathcal{K}(n,k)$.
\end{remark}

\noindent\begin{theorem}\label{th:c-5} Let $n,k,m$ be positive integers
such that $1\leq k \leq n-2$ and $1\leq m \leq n-k-1$. Then
$$\resizebox{.9\hsize}{!}
{$\mu_\alpha(K(n,k,m))= \frac{\alpha m+\alpha n+n-2+\sqrt{
(1-\alpha)^2n^2+(2\alpha^2-6\alpha+4)mn+(\alpha^2+4\alpha-4)m^2-4(1-\alpha)km}}{2}$}.$$
\end{theorem}

\begin{proof} Let $G=K(n,k,m)$, and $S$ be a $k$-vertex cut of
$G$. Suppose that $G_1$ with $|V(G_1)|=m$ and $G_2$ with $|V(G_2)|=n-k-m$
are two strongly connected components, i.e., two complete subdigraphs of $G-S$
with arcs $\{(u,v): u\in V(G_1), v\in V(G_2)\}$. Let $X$ be the Perron vector of $D_\alpha(G)$. It is easy to know
that all coordinates of the Perron vector of $D_\alpha(G)$ corresponding to vertices $V(G_1)\cup S$ are equal, say $x_1$,
all coordinates corresponding to vertices $V(G_2)$ are equal, say $x_2$.
Therefore, we get
$$
\begin{cases}
\ \mu_\alpha(G) x_1=\alpha(n-1)x_1+(1-\alpha)(k+m-1)x_1+(1-\alpha)(n-k-m)x_2,\\
\ \mu_\alpha(G) x_2=\alpha(n+m-1)x_2+(1-\alpha)kx_1+2(1-\alpha)mx_1+(1-\alpha)(n-k-m-1)x_2.
\end{cases}$$
Or equivalently
$$\left(
  \begin{array}{ccc}
  \alpha(n-k-m)+k+m-1 & (1-\alpha)(n-k-m) \\
   (1-\alpha)(k+2m) & \alpha(k+2m)+n-k-m-1
  \end{array}
\right)
\left(
          \begin{array}{c}
            x_1 \\
            x_2 \\
          \end{array}
        \right)=\mu_\alpha(G) \left(
          \begin{array}{c}
            x_1 \\
            x_2 \\
          \end{array}
        \right)$$
Let $f(x)=x^2-(\alpha n+\alpha m+n-2)x+1-n-mn-\alpha km+2\alpha nm-\alpha n+\alpha n^2-\alpha m-\alpha m^2+km+m^2$.
It is easy to know that $\mu_\alpha(G)$ is the largest real root of the equation
$f(x)=0$, where $1\leq m\leq n-k-1$. Since the above $2\times2$ matrix is nonnegative irreducible, $\mu_\alpha(G)$ is
an eigenvalue of the above $2\times2$ matrix with multiplicity 1. Then the discriminant of $f(x)$ is greater than 0.
Therefore, we have
$$\resizebox{.9\hsize}{!}
{$\mu_\alpha(K(n,k,m))= \frac{\alpha m+\alpha n+n-2+\sqrt{
(1-\alpha)^2n^2+(2\alpha^2-6\alpha+4)mn+(\alpha^2+4\alpha-4)m^2-4(1-\alpha)km}}{2}$}.$$
\end{proof}

\noindent\begin{remark} Note that $\overset{\longleftrightarrow}{K_{n}}$ is the unique
digraph which achieves the minimum $D_\alpha(G)$ spectral radius $n-1$ among all strongly connected
digraphs, and $K(n,n-2,1)=\overset{\longleftrightarrow}{K_{n}}-\{(u,v)\}$
where $u,v\in V(\overset{\longleftrightarrow}{K_{n}})$, by Corollary \ref{co:1}
and Theorem \ref{th:c-5}, we deduce that $K(n,n-2,1)$ is the unique
digraph which achieves the second minimum $D_\alpha(G)$ spectral radius
$$\frac{n+\alpha n+\alpha-2+\sqrt{
(1-\alpha)^2n^2-2\alpha(1-\alpha)n+\alpha^2-4\alpha+4}}{2}$$ among all
strongly connected digraphs of order $n$.
\end{remark}

\noindent\begin{theorem}\label{th:c-6} Let $n,k$ be positive integers
such that $1\leq k \leq n-2$, $G\in\mathcal{G}_{n,k}$. Then we have

$(i)$ For $\alpha=0$, $\mu_\alpha(G)\geq\frac{n-2+\sqrt{n^2+4n-4k-4}}{2}$ with equality if and only if
$G\cong K(n,k,n-k-1)$ or $G\cong K(n,k,1)$.

$(ii)$ For $0<\alpha\leq\frac{4}{5}$, $\mu_\alpha(G)\geq\frac{n-2+\alpha+\alpha n+\sqrt{
(1-\alpha)^2n^2+(2\alpha^2-6\alpha+4)n+(\alpha^2+4\alpha-4)-4(1-\alpha)k}}{2},$
with equality if and only if $G\cong K(n,k,1)$.
\end{theorem}

\begin{proof} By Remark \ref{re:1}, $\mu_\alpha(G)\geq \mu_\alpha(K(n,k,m))$ for some $m$, where $1\leq m \leq n-k-1$. By Theorem \ref{th:c-5}, we have
$$\resizebox{.9\hsize}{!}
{$\mu_\alpha(K(n,k,m))= \frac{\alpha m+\alpha n+n-2+\sqrt{
(1-\alpha)^2n^2+(2\alpha^2-6\alpha+4)mn+(\alpha^2+4\alpha-4)m^2-4(1-\alpha)km}}{2}$}.$$
 Now we want to show that the minimum value of $\mu_\alpha(K(n,k,m))$ must be
taken at either $m=1$ or at $m=n-k-1$.

Let $\resizebox{.9\hsize}{!}
{$g(m)=n-2+\alpha m+\alpha n+\sqrt{
(1-\alpha)^2n^2+(2\alpha^2-6\alpha+4)mn+(\alpha^2+4\alpha-4)m^2-4(1-\alpha)km}$}$.
\\Then
$$g(m)'=\alpha+\frac{1}{2}\frac{(2\alpha^2-6\alpha+4)n+2(\alpha^2+4\alpha-4)m-4(1-\alpha)k}{\sqrt{
(1-\alpha)^2n^2+(2\alpha^2-6\alpha+4)mn+(\alpha^2+4\alpha-4)m^2-4(1-\alpha)km}},$$
\begin{align*}
g(m)''&=\frac{1}{4}\frac{16k^2(2\alpha-\alpha^2-1)+16nk(-5\alpha+4\alpha^2+2-\alpha^3)+32n^2(3\alpha-3\alpha^2+\alpha^3-1)}{(
(1-\alpha)^2n^2+(2\alpha^2-6\alpha+4)mn+(\alpha^2+4\alpha-4)m^2-4(1-\alpha)km)^{\frac{3}{2}}}\\
&=\frac{1}{4}\frac{-16k^2(\alpha-1)^2+16nk(\alpha-1)^2(2-\alpha)+32n^2(\alpha-1)^3}{(
(1-\alpha)^2n^2+(2\alpha^2-6\alpha+4)mn+(\alpha^2+4\alpha-4)m^2-4(1-\alpha)km)^{\frac{3}{2}}}\\
&=\frac{1}{4}\frac{16(\alpha-1)^2(2n^2(\alpha-1)-k^2+nk(2-\alpha))}{(
(1-\alpha)^2n^2+(2\alpha^2-6\alpha+4)mn+(\alpha^2+4\alpha-4)m^2-4(1-\alpha)km)^{\frac{3}{2}}}.
\end{align*}
Take $f(\alpha)=2n^2(\alpha-1)-k^2+nk(2-\alpha)=(2n^2-nk)\alpha-2n^2+2nk-k^2$. Then $f(\alpha)<0$ for all
$\alpha<\frac{2n^2-2nk+k^2}{2n^2-nk}$. Since $\frac{2n^2-2nk+k^2}{2n^2-nk}>\frac{4}{5}$, $f(\alpha)<0$ for all $0\leq\alpha\leq\frac{4}{5}$.
Hence, $g(m)''<0$ for all $0\leq\alpha\leq\frac{4}{5}$. Thus, for fixed $n$ and $k$, the minimum value of $g(m)$ must be
taken at either $m=1$ or at $m=n-k-1$.

In the following, we want to compare $g(1)$ and $g(n-k-1)$.
Let $B=(\alpha^2+1-2\alpha)n^2+(2\alpha^2-6\alpha+4)n-4+4\alpha+\alpha ^2+4\alpha k-4k$
and $A=(4\alpha^2-4\alpha+1)n^2+(2\alpha k-4\alpha^2k-2\alpha-4\alpha^2+4)n-4+4\alpha+\alpha ^2+4\alpha k+\alpha^2k^2+2\alpha^2k-4k$. Then
\begin{align*}
g(n-k-1)-g(1)&=\alpha n-\alpha k-2\alpha+\sqrt{A}-\sqrt{B}\\
&=\alpha(n-k-2)+\frac{A-B}{\sqrt{A}+\sqrt{B}}\\
&=\alpha(n-k-2)+\frac{\alpha(n-k-2)(-2n+3\alpha n-\alpha k)}{\sqrt{A}+\sqrt{B}}\\
&=\alpha(n-k-2)(1+\frac{-2n+3\alpha n-\alpha k}{\sqrt{A}+\sqrt{B}}).
\end{align*}

For $\alpha=0$, we have $g(n-k-1)-g(1)=0$, that is $\mu_\alpha(K(n,k,n-k-1))=\mu_\alpha(K(n,k,1))=\frac{n-2+\sqrt{n^2+4n-4k-4}}{2}$.
 Therefore, $\mu_\alpha(G)\geq\frac{n-2+\sqrt{n^2-4n+4k+4}}{2}$ with equality if and only if
$G\cong K(n,k,n-k-1)$ or $G\cong K(n,k,1)$.

For $0<\alpha<1$. We assume that $n>k+2$ since in case $n=k+2$ there is only one value of
$m$ under consideration. Now suppose that $g(n-k-1)-g(1)\leq0$. We will deduce a contradiction. We have simultaneously
$$\sqrt{A}+\sqrt{B}\leq 2n-3\alpha n+\alpha k\ \textrm{ and}  \ \sqrt{A}-\sqrt{B}\leq2\alpha+\alpha k-\alpha n.$$
So $\sqrt{A}\leq n+\alpha k-2\alpha n+\alpha$.
However, $A-(n+\alpha k-2\alpha n+\alpha)^2=-4\alpha n+4n-4+4\alpha k+4\alpha-4k=4(1-\alpha)(n-k-1)>0$, that is $A>(n+\alpha k-2\alpha n+\alpha)^2$.
Thus $g(n-k-1)-g(1)>0$. Then $\mu_\alpha(K(n,k,n-k-1))>\mu_\alpha(K(n,k,1))$.

Therefore, for $0<\alpha\leq\frac{4}{5}$, $\mu_\alpha(G)\geq\frac{n-2+\alpha+\alpha n+\sqrt{
(1-\alpha)^2n^2+(2\alpha^2-6\alpha+4)n+(\alpha^2+4\alpha-4)-4(1-\alpha)k}}{2}$, with equality if and only if $G\cong K(n,k,1)$.

Hence, we get the desired result.
\end{proof}

\noindent\begin{theorem}\label{th:c-7} Let $n,k$ be positive integers
such that $n\geq2k+2$, $G\in\mathcal{G}_{n,k}$. Then we have

$(i)$ For $\alpha=0$, $\mu_\alpha(G)\geq\frac{n-2+\sqrt{n^2+4n-4k-4}}{2}$ with equality if and only if
$G\cong K(n,k,n-k-1)$ or $G\cong K(n,k,1)$.

$(ii)$ For $0<\alpha<1$, $\mu_\alpha(G)\geq\frac{n-2+\alpha+\alpha n+\sqrt{
(1-\alpha)^2n^2+(2\alpha^2-6\alpha+4)n+(\alpha^2+4\alpha-4)-4(1-\alpha)k}}{2},$
with equality if and only if $G\cong K(n,k,1)$.
\end{theorem}

\begin{proof} If $0\leq\alpha\leq\frac{4}{5}$, then by Theorem \ref{th:c-6}, we get the desired result. Therefore, we only consider the case
$\frac{4}{5}<\alpha<1$ in the following. By Remark \ref{re:1}, $\mu_\alpha(G)\geq \mu_\alpha K(n,k,m)$ for some $m$, where $1\leq m \leq n-k-1$.
By Theorem \ref{th:c-5}, we have $\mu_\alpha(K(n,k,1))$ is the largest real root of the equation
$x^2-(\alpha n+\alpha +n-2)x+2-2n-\alpha k+\alpha n+\alpha n^2-2\alpha+k=0$.
Let $f(x)=x^2-(\alpha n+\alpha +n-2)x+2-2n-\alpha k+\alpha n+\alpha n^2-2\alpha+k$ with axis of symmetry
$\widetilde{x}=\frac{\alpha n+\alpha +n-2}{2}<n.$ Then $f(x)\geq f(n)=2-\alpha k+k-2\alpha>0$ for all $x\geq n$.
Hence $\mu_\alpha(K(n,k,1))<n$. In the following, we want to prove $\mu_\alpha(K(n,k,m))>n$ for $2\leq m\leq n-k-1$.

For $2\leq m\leq n-k-1$, $\mu_\alpha(K(n,k,m))$ is the largest root of the equation
$x^2-(\alpha n+\alpha m+n-2)x+1-n-mn-\alpha km+2\alpha nm-\alpha n+\alpha n^2-\alpha m-\alpha m^2+km+m^2=0$.
Let $h(x)=x^2-(\alpha n+\alpha m+n-2)x+1-n-mn-\alpha km+2\alpha nm-\alpha n+\alpha n^2-\alpha m-\alpha m^2+km+m^2$.
Then $h(n)=(1-\alpha)m^2+(\alpha n-\alpha-\alpha k+k-n)m+n+1-an$. Take $p(m)=(1-\alpha)m^2+(\alpha n-\alpha-\alpha k+k-n)m+n+1-an$.
Since $p''(m)>0$, $p(m)\leq \max\{p(2),p(n-k-1)\}=\max\{(n-2k-6)\alpha+5-n+2k,2-\alpha n+k\}$. We discuss the following two cases.

{\bf Case 1.} If $2k+2\leq n<2k+6$, then $n-2k-6<0$, $(n-2k-6)\alpha+5-n+2k<(n-2k-6)\frac{4}{5}+5-n+2k=-\frac{1}{5}(n-2k-1)<0$,
$2-\alpha n+k<2+k-\frac{4}{5}n\leq-\frac{1}{5}(3k-2)<0$. Thus $p(m)\leq \max\{p(2),p(n-k-1)\}<0$.
Hence $h(n)<0$. Therefore, $\mu_\alpha(K(n,k,m))>n$ for $2\leq m\leq n-k-1$.

{\bf Case 2.} If $n\geq2k+6$. Then $n-2k-6\geq0$, $(n-2k-6)\alpha+5-n+2k\leq(n-2k-6)+5-n+2k=-1<0$,
$2-\alpha n+k<2+k-\frac{4}{5}n\leq-\frac{1}{5}(3k+14)<0$. Hence $h(n)<0$. Therefore, $\mu_\alpha(K(n,k,m))>n$ for $2\leq m\leq n-k-1$.

Combining the above two cases, we have $\mu_\alpha(K(n,k,m))>n>\mu_\alpha(K(n,k,1))$ for $2\leq m\leq n-k-1$ and $\frac{4}{5}<\alpha<1$.
Therefore, for $\frac{4}{5}<\alpha<1$
$\mu_\alpha(G)\geq K(n,k,1)$ with equality if and only if $G\cong K(n,k,1)$.

Hence, we get the desired result.
\end{proof}

For general case, we propose the following conjecture based on numerical examples.

\noindent\begin{conjecture} \ Let $n,k$ be positive integers, $0<\alpha<1$, $G\in\mathcal{G}_{n,k}$. Then
 $\mu_\alpha(G)\geq\mu_\alpha(K(n,k,1))$,
with equality if and only if $G\cong K(n,k,1)$.
\end{conjecture}

\section{The minimum $D_\alpha$ spectral radius of strongly connected digraphs with given arc connectivity}

Let $\mathcal{G}^*_{n,k}$ denote the set of strongly connected digraphs
with order $n$ and arc connectivity $\kappa'(G)=k\geq1$. If $k=n-1$,
then $\mathcal{G}^*_{n,k}=\{\overset{\longleftrightarrow}{K_{n}}\}$.
So we only consider the cases $1\leq k \leq n-2$.

In \cite{LiS}, Lin and Shu proved that $K(n,k,n-k-1)$ or $K(n,k,1)$ attains the maximum $D_0$ spectral radius among all strongly connected
digraphs with given arc connectivity. We generalize their results to $0\leq\alpha<1$.

\noindent\begin{lemma}\label{le:12} (\cite{XiWa3}) Let $G$ be a strongly connected digraph
with order $n$ and arc connectivity $k\geq1$, and $S$ be an arc cut set of $G$ of size $k$
such that $G-S$ has exactly two strongly connected components, say $G_1$ and $G_2$ with $|V(G_1)|=n_1$ and $|V(G_2)|=n_2$, where $n_1+n_2=n$.
If $d_v^+>k$ and $d_v^->k$ for each vertex $v\in V(G)$, then $n_1\geq k+2$, $n_2\geq k+2$.
\end{lemma}

\noindent\begin{lemma}\label{le:13} Let $G\in\mathcal{G}^*_{n,k}$, which contains a vertex of outdegree $k$. Then
$$\mu_\alpha(G)\geq \mu_\alpha(K(n,k,n-k-1)).$$
\end{lemma}

\begin{proof} Let $w$ be a vertex of $G$ such that $d_w^+=k$. Then the arcs out-incident to $w$ form an arc cut set of size $k$. Adding all possible arcs from $G\setminus\{w\}$ to $G\setminus\{w\}\cup\{w\}$, we obtain a digraph $H$, which is isomorphic to $K(n,k,n-k-1)$, the arc connectivity of $H$ remains equal to $k$. If
$G\neq K(n,k,n-k-1)$, then $\mu_\alpha(G)>\mu_\alpha(K(n,k,n-k-1))$ by Corollary \ref{co:1}. So the result follows.
\end{proof}

\noindent\begin{lemma}\label{le:14} Let $G\in\mathcal{G}^*_{n,k}$, which contains a vertex of indegree $k$. Then
$$\mu_\alpha(G)\geq \mu_\alpha(K(n,k,1)).$$
\end{lemma}

\begin{proof} Let $w$ be a vertex of $G$ such that $d_w^-=k$. Then the arcs in-incident to $w$ form an arc cut set of size $k$. Adding all possible arcs from $w$ to $G\setminus\{w\}$, and all possible arcs from $G\setminus\{w\}$ to $G\setminus\{w\}$, we obtain a digraph $H'$, which is isomorphic to $K(n,k,1)$, the arc connectivity of $H'$ remains equal to $k$. If
$G\neq K(n,k,1)$, then $\mu_\alpha(G)>\lambda_\alpha(K(n,k,1))$ by Corollary \ref{co:1}. So the result follows.
\end{proof}

\noindent\begin{theorem}\label{th:c-9}
Let $G\in\mathcal{G}^*_{n,k}$. Then we have

$(i)$ For $\alpha=0$, $\mu_\alpha(G)\geq\mu_\alpha(K(n,k,1))=\mu_\alpha(K(n,k,n-k-1))$ with equality if and only if
$G\cong K(n,k,n-k-1)$ or $G\cong K(n,k,1)$.

$(ii)$ For $0<\alpha<1$, $\mu_\alpha(G)\geq\mu_\alpha(K(n,k,1))$, with equality if and only if $G\cong K(n,k,1)$.
\end{theorem}

\begin{proof} Let $G$ be a digraph in $\mathcal{G}^*_{n,k}$. Note that each vertex in the digraph $G$ has outdegree at least
$k$ and indegree at least $k$, otherwise $G\notin\mathcal{G}^*_{n,k}$. Then, we consider the following two cases.

{\bf Case 1.} If there exists a vertex $u$ of $G$ with outdegree $k$, by Lemma \ref{le:13},
$\mu_\alpha(G)\geq\mu_\alpha(K(n,k,n-k-1))$. If there exists a vertex $u$ of $G$ with indegree $k$, by Lemma \ref{le:14},
$\mu_\alpha(G)\geq\mu_\alpha(K(n,k,1))$. However, by the proof of Theorem \ref{th:c-6}, we have
$\mu_\alpha(K(n,k,1))=\mu_\alpha(K(n,k,n-k-1))$ for $\alpha=0$, and $\mu_\alpha(K(n,k,n-k-1))>\mu_\alpha(K(n,k,1))$ for $0<\alpha<1$.
The result follows in this case.

{\bf Case 2.} We suppose that all
vertices of $G$ have outdegree greater than $k$ and indegree greater than $k$. Let $S$ be an
arc cut set of $G$ containing $k$ arcs, then $G-S$ consists of exactly two strongly
connected components $G_1$, $G_2$, with order $n_1$, $n_2$, respectively. Without loss of generality,
we may assume that there are no arcs from $G_1$ to $G_2$ in $G-S$.
By Lemma \ref{le:12}, $n_1\geq k+2$, $n_2=n-n_1\geq k+2$, then $k+2\leq n_1\leq n-k-2$, $n\geq n_1+k+2\geq2k+4$.
Next we construct a new digraph $G'$ by
adding to $G$ any possible arcs from $G_2$
to $G_1\cup G_2$ or any possible arcs from $G_1$ to $G_1$ that were not present in $G$. Obviously,
the arc connectivity of $G'$ remains equal to $k$ and all
vertices of $G'$ have outdegree greater than $k$ and indegree still greater than $k$. By  Corollary \ref{co:1},
the addition of any such arc will give $\mu_\alpha(G)>\mu_\alpha(G')$.
Let $G''=\overset{\longleftrightarrow}{K_{n_1}}\cup\overset{\longleftrightarrow}{K_{n_2}}$, $U=\{u_1, u_2, \cdots, u_k\}$ be a set of $k$ vertices in
$V(\overset{\longleftrightarrow}{K_{n_1}})$ and $W=\{v_1, v_2, \cdots, v_k\}$ be a set of $k$ vertices in $V(\overset{\longleftrightarrow}{K_{n_2}})$. Let
$H_0$ be a digraph obtained from $G''$ by adding all possible arcs from $U$ to $W$, and adding all possible arcs from $\overset{\longleftrightarrow}{K_{n_2}}$ to $\overset{\longleftrightarrow}{K_{n_1}}$. Note that $G'$ is a spanning strongly connected subdigraph of $H_0$, therefore,
by Corollary \ref{co:1}, $\mu_\alpha(G')\geq\mu_\alpha(H_0)$.
However, we can know that $n\geq2k+4$, the vertex connectivity of $H_0$ is $k$, $H_0\ncong K(n,k,1)$ and $H_0\ncong K(n,k,n-k-1)$. Hence, by Theorem \ref{th:c-7},
we know that
$\mu_\alpha(H_0)>\mu_\alpha(K(n,k,1))$. Therefore, $\mu_\alpha(G)\geq\mu_\alpha(G')\geq\mu_\alpha(H_0)>\mu_\alpha(K(n,k,1))$.

Therefore, combining the above two cases, we get the desired result.
\end{proof}

\end{document}